\documentclass[11pt]{article}

\usepackage{amsmath,amssymb,amsthm,mathtools}
\usepackage[margin=1in]{geometry}
\usepackage{enumitem}

\newtheorem{theorem}{Theorem}
\newtheorem{proposition}{Proposition}
\newtheorem{lemma}{Lemma}
\newtheorem{remark}{Remark}
\theoremstyle{definition}
\newtheorem{property}{Property}

\numberwithin{equation}{section}

\newcommand{\R}{\mathbb R}
\newcommand{\C}{\mathbb C}
\newcommand{\F}{\mathcal F}
\newcommand{\Chat}{\widehat C}
\newcommand{\norm}[1]{\left\lVert #1\right\rVert}
\newcommand{\rank}{\operatorname{rank}}
\newcommand{\Ran}{\operatorname{Ran}}

\title{Endpoint Asymptotics of Optimal Stabilization Prefactors}
\author{Changqin Quan\\
Graduate School of System Informatics, Kobe University, Kobe, Japan
\and
Gengsheng Wang\\
Hetao Institute of Mathematics and Interdisciplinary Studies, Shenzhen, China}
\date{}

\begin{document}
\maketitle
\begin{abstract}
Consider the stabilizable finite-dimensional linear control system
$\dot x=Ax+Bu$.
For a prescribed decay rate $\delta>0$, we define the optimal stabilization
prefactor
$$
\Chat(\delta)
:=
\inf_{K\in\R^{m\times n}}
\sup_{t\ge0}e^{\delta t}\norm{e^{(A+BK)t}}.
$$
We determine its asymptotic behavior as $\delta$ approaches the right
endpoint of the achievable decay-rate range. If $(A,B)$ is controllable
and $\mu$ is its largest controllability index, then
$\Chat(\delta)\asymp\delta^{\mu-1}$ as $\delta\to+\infty$; whereas if
$(A,B)$ is stabilizable but not controllable, then
$\Chat(\delta)\asymp(\delta_*-\delta)^{-(q-1)}$ as
$\delta\uparrow\delta_*$, where $\delta_*$ is the supremal achievable
decay rate and $q$ is the largest size of a Jordan block of the
uncontrollable part associated with the spectral boundary
$\operatorname{Re}\lambda=-\delta_*$.
Thus, in both cases, the endpoint asymptotic order of $\Chat(\delta)$ is
determined by the corresponding structural invariant---$\mu$ in the
controllable case and $q$ in the noncontrollable case---and conversely this
order recovers that invariant.
\end{abstract}

\section{Introduction}
\label{sec:introduction}

Consider the finite-dimensional linear control system
\begin{equation}\label{eq:system}
    \dot x(t)=Ax(t)+Bu(t),
    \qquad
    A\in\R^{n\times n},\quad B\in\R^{n\times m},
    \qquad t\ge0,
\end{equation}
where $x(t)\in\R^n$ is the state and $u(t)\in\R^m$ is the control.
We assume throughout that $(A,B)$ is stabilizable, endow $\R^n$ with the
Euclidean norm, and use the same notation $\norm{\cdot}$ for the induced
operator norm on matrices.
\vskip 5pt
\noindent{\bf Motivation and significance} Under a constant state feedback
$u=Kx$, an exponential stabilization estimate has the form
\begin{equation}\label{eq:intro-estimate}
    \norm{e^{(A+BK)t}}\le C e^{-\delta t},
    \qquad t\ge0,
\end{equation}
where $\delta>0$ is the prescribed decay rate and $C\ge1$ is the
corresponding transient prefactor. The two quantities describe different
features of the closed-loop dynamics: $\delta$ governs the asymptotic rate
of decay, whereas $C$ controls the transient amplification compatible with
that rate. Pole placement and high-gain feedback address how rapidly the
closed-loop system can be made to decay, but they do not determine the
optimal transient prefactor associated with a prescribed decay rate. This
leads naturally to the following definition.

For $\delta>0$, we define the \emph{optimal stabilization prefactor}
\begin{equation}\label{eq:Chat}
    \Chat(\delta)
    :=
    \inf_{K\in\R^{m\times n}}
    \sup_{t\ge0}e^{\delta t}\norm{e^{(A+BK)t}},
    \qquad \delta>0.
\end{equation}
For $\delta>0$ and $K\in\R^{m\times n}$, set
\begin{equation}\label{eq:CdeltaK}
    C(\delta,K)
    :=
    \sup_{t\ge0}e^{\delta t}\norm{e^{(A+BK)t}}
    =
    \sup_{t\ge0}\norm{e^{(A+BK+\delta I_n)t}}
    \in[1,+\infty].
\end{equation}
Then
\begin{equation}\label{eq:Chat-via-CdeltaK}
    \Chat(\delta)
    =
    \inf_{K\in\R^{m\times n}}C(\delta,K).
\end{equation}
Thus, for a prescribed decay rate $\delta$, $\Chat(\delta)$ is the optimal
transient prefactor over all constant state feedbacks. When finite,
$C(\delta,K)$ is the least constant for which
\eqref{eq:intro-estimate} holds for the fixed feedback $K$. Moreover,
\eqref{eq:CdeltaK} shows that $C(\delta,K)$ is precisely the maximum
transient amplification of the shifted closed-loop matrix
$A+BK+\delta I_n$.

Whidborne and McKernan~\cite{WhidborneMcKernan2007} studied maximum
transient energy growth and its reduction by feedback. If the state energy
is measured by $\norm{x}^2$ and
$G(M):=\sup_{t\ge0}\norm{e^{Mt}}^2$, then
$G(A+BK+\delta I_n)=C(\delta,K)^2$. More recently, Apkarian and
Noll~\cite{ApkarianNoll2025} studied closed-loop transient mitigation by
optimizing a Kreiss-system-norm criterion, which is related to but different
from $C(\delta,K)$. Thus these works concern $C(\delta,K)$ or related
transient measures, but not the feedback-optimized quantity
$\Chat(\delta)$. In particular,
neither work studies the function $\delta\mapsto\Chat(\delta)$ or its
asymptotic behavior at the right endpoint of the achievable decay-rate
range.

The optimization over all constant state feedbacks is essential in our
problem. For each prescribed decay rate $\delta$, the quantity
$\Chat(\delta)$ gives the best transient prefactor achievable by constant
state feedback.

Equally important is the behavior of this optimal value as $\delta$
approaches the right endpoint of the achievable decay-rate range. In this
endpoint limit, $\Chat(\delta)$ may either remain bounded or blow up. In the
latter case, its asymptotic behavior quantifies the rate of blow-up. Since
the infimum in \eqref{eq:Chat} is taken over all constant state feedbacks,
with no constraint on the feedback gain, this endpoint behavior is an
intrinsic feature of the feedback stabilization problem, rather than an
artifact of a particular feedback construction.

We further ask what structural information is encoded in this endpoint
behavior. More specifically, we ask whether it is determined by intrinsic
invariants of the pair $(A,B)$ and, conversely, whether these invariants can
be recovered from the optimal stabilization performance. For a controllable
pair, the right endpoint of the achievable decay-rate range is $+\infty$;
for a stabilizable but uncontrollable pair, it is finite and is determined
by the uncontrollable dynamics. These two endpoint regimes lead naturally
to two separate analyses.
\vskip 5pt

\noindent {\bf Main results} We begin with the controllable case. Define
\begin{equation}\label{eq:mu-rank}
    \mu
    :=
    \min\left\{
    j\ge1:
    \rank[B,AB,\ldots,A^{j-1}B]=n
    \right\}.
\end{equation}
As recalled in Subsection~\ref{sec:controllability-indices}, $\mu$ is the
largest controllability index of $(A,B)$, equivalently the length of its
longest controllability chain.

\begin{theorem}
\label{thm:main-asymptotic}
Assume that $(A,B)$ is controllable. Let $\mu$ be the largest controllability
index of $(A,B)$, defined in \eqref{eq:mu-rank}. Then there exist constants
$c_0,C_0,\delta_0>0$, depending only on $(A,B)$, such that
\begin{equation}\label{eq:main-asymptotic}
c_0\,\delta^{\mu-1}
\le
\Chat(\delta)
\le
C_0\,\delta^{\mu-1},
\qquad
\delta\ge\delta_0.
\end{equation}
Moreover,
\begin{equation}\label{eq:log-slope}
\lim_{\delta\to+\infty}
\frac{\log\Chat(\delta)}{\log\delta}
=
\mu-1.
\end{equation}
If $\mu=1$, equivalently $\rank B=n$, then
\begin{equation}\label{eq:mu-one-exact}
\Chat(\delta)=1,
\qquad \delta>0.
\end{equation}
\end{theorem}

\begin{remark}
$(i)$
To the best of our knowledge, Theorem~\ref{thm:main-asymptotic} is new. It
determines the asymptotic behavior of the optimal stabilization prefactor
after optimization over all constant state feedbacks. Since
\eqref{eq:main-asymptotic} is two-sided, $\delta^{\mu-1}$ is the sharp
growth order and cannot be improved by any constant state feedback. More
significantly, \eqref{eq:main-asymptotic} and \eqref{eq:log-slope} together
establish a two-way characterization: the largest controllability index
$\mu$, an algebraic feedback invariant, determines the endpoint asymptotic
order of $\Chat(\delta)$, and conversely this order recovers $\mu$.

$(ii)$
When $\mu\ge2$, the lower bound in
Theorem~\ref{thm:main-asymptotic} is based on a feedback-invariant relation
along a longest Brunovsk\'y chain. For the closed-loop matrix in Brunovsk\'y
coordinates, this relation yields a feedback-independent identity for its
$(\mu-1)$st negative power. The Laplace representation of this negative
power, combined with the prescribed decay estimate, then yields the lower
bound of order $\delta^{\mu-1}$ for $C(\delta,K)$. The case $\mu=1$ is
covered by the exact identity \eqref{eq:mu-one-exact}.
\end{remark}

The controllable case corresponds to an infinite achievable-rate endpoint.
We next turn to the stabilizable but uncontrollable case. Here the optimal
stabilization prefactor has a finite endpoint, and its endpoint asymptotics
are governed by a structural invariant of the uncontrollable dynamics. Let
$A_{\rm u}$ denote the uncontrollable block in an orthogonal
Kalman decomposition and set
\begin{equation}\label{eq:deltastar}
    \delta_*
    :=
    -\max\{\operatorname{Re}\lambda:\lambda\in\sigma(A_{\rm u})\}>0.
\end{equation}
The spectrum and Jordan block sizes of $A_{\rm u}$ do not depend on the
choice of orthogonal Kalman decomposition; this standard fact is recalled in
Section~\ref{sec:finite}.

\begin{theorem}
\label{thm:finite-endpoint}
Assume that $(A,B)$ is stabilizable but not controllable. Let $q$ be the
largest size of a Jordan block in the complex Jordan form of $A_{\rm u}$
associated with an eigenvalue $\lambda$ satisfying
$\operatorname{Re}\lambda=-\delta_*$, where $\delta_*$ is defined in
\eqref{eq:deltastar}. Then there exist constants
$c_0,C_0,\varepsilon_0>0$, depending only on $(A,B)$, such that
\begin{equation}\label{eq:finite-endpoint-asymptotic}
    c_0(\delta_*-\delta)^{-(q-1)}
    \le
    \Chat(\delta)
    \le
    C_0(\delta_*-\delta)^{-(q-1)},
    \qquad 0<\delta_*-\delta\le\varepsilon_0.
\end{equation}
Moreover,
\begin{equation}\label{eq:recover-q}
    q
    =
    1+
    \lim_{\delta\uparrow\delta_*}
    \frac{\log\Chat(\delta)}
    {\log\bigl(1/(\delta_*-\delta)\bigr)}.
\end{equation}
Finally, the endpoint $\delta_*$ itself is achievable if and only if $q=1$;
equivalently,
\begin{equation}\label{eq:q-one-endpoint}
    \Chat(\delta_*)<+\infty
    \quad\Longleftrightarrow\quad
    q=1.
\end{equation}
\end{theorem}

\begin{remark}\label{remark1-8-29}
$(i)$
To the best of our knowledge, Theorem~\ref{thm:finite-endpoint} is also new.
Although it complements Theorem~\ref{thm:main-asymptotic} by completing the
endpoint picture for stabilizable finite-dimensional systems, it has
independent structural significance. The critical Jordan-block size $q$
determines the endpoint asymptotic order of $\Chat(\delta)$, and conversely
\eqref{eq:recover-q} shows that this order recovers $q$. In addition,
\eqref{eq:q-one-endpoint} shows that $q=1$ exactly characterizes the case
in which the endpoint $\delta_*$ itself is achievable.

$(ii)$
The lower bound in Theorem~\ref{thm:finite-endpoint} is obtained by first
isolating the feedback-invariant uncontrollable block through the Kalman
decomposition. After shifting by $\delta_*I_{n_{\rm u}}$, a critical Jordan
chain of length $q$ produces polynomial semigroup growth of order
$t^{q-1}$. When
$q\ge2$, with $\varepsilon=\delta_*-\delta$, balancing this polynomial
growth against the exponential factor $e^{-\varepsilon t}$ at the time
scale $t\sim\varepsilon^{-1}$ yields the lower bound of order
$\varepsilon^{-(q-1)}$. For $q=1$, the corresponding lower bound is of
order one.

$(iii)$
Together, Theorems~\ref{thm:main-asymptotic} and
\ref{thm:finite-endpoint} give a complete description of the endpoint
behavior of the optimal stabilization prefactor for stabilizable
finite-dimensional pairs. Their lower bounds arise from different
mechanisms: a feedback-invariant relation along a longest controllability
chain in the controllable case, and polynomial growth generated by a
critical uncontrollable Jordan chain in the finite-endpoint case.
\end{remark}

\noindent{\bf Related work}
Pole-placement overshoot estimates provide a closely related line of work.
An earlier estimate was obtained by Sun, Wang, Xie, and Yu~\cite{SWXY2006}.
Most recently, Li, Zhou, and Cao~\cite{LiZhouCao2026} obtained bounds for a
parametric Lyapunov equation and, as an application, a pole-placement
overshoot estimate with exponent $\mu-1$. Their work does not study the
feedback-optimized quantity $\Chat(\delta)$. However, in the present
notation, their estimate gives, for a particular family of feedbacks
$K_\delta$,
$C(\delta,K_\delta)\le C\delta^{\mu-1}$, and therefore implies the
same-order upper bound for $\Chat(\delta)$. Theorem~\ref{thm:main-asymptotic}
further provides the matching lower bound for $\Chat(\delta)$, showing that
the order $\delta^{\mu-1}$ is optimal over all constant state feedbacks.
Our proof of the upper bound is independent of \cite{LiZhouCao2026} and
uses the scaling structure of the Brunovsk\'y form.

A related structural phenomenon appears in fast-control problems. In
Seidman's result~\cite{Seidman1988}, the relevant controllability-layer
index is $\mu-1$, while the worst-case minimum $L^2$-control norm has order
$T^{-(\mu-1/2)}$ as $T\downarrow0$. Seidman and
Yong~\cite{SeidmanYong1996} extended this analysis to $L^p$ controls.
Although these open-loop fast-control problems are different from the
feedback-optimized problem studied here, both are governed by the same
longest controllability chain.

Prescribed transient bounds for feedback stabilization form another closely
related line of work. Hinrichsen, Plischke, and Wirth~\cite{HPW2002} and
Plischke and Wirth~\cite{PW2004} studied, in the present notation, the
feasibility problem
\[
    \text{given }(M,\delta),\qquad
    \text{does there exist }K\text{ such that }C(\delta,K)\le M?
\]
Their work does not introduce or study the feedback-optimized quantity
\[
    \Chat(\delta)=\inf_{K\in\R^{m\times n}}C(\delta,K),
\]
nor its endpoint asymptotics. The latter is the problem addressed in the
present paper.

The paper is organized as follows. Section~\ref{sec:preliminaries} records
the auxiliary quantities and controllability properties used in the
analysis. Sections~\ref{sec:controllable} and \ref{sec:finite} prove
Theorems~\ref{thm:main-asymptotic} and \ref{thm:finite-endpoint},
respectively. Section~\ref{sec:examples} illustrates the structural content
of the two results, and Section~\ref{sec:appendix} contains auxiliary
proofs.

\section{Preliminaries}
\label{sec:preliminaries}

This section introduces the auxiliary quantities used in the analysis of
both endpoint problems and records the controllability properties needed
for the proof of Theorem~\ref{thm:main-asymptotic}.

\subsection{The optimal stabilization prefactor and the achievable decay-rate range}
\label{sec:setting}

Recall the quantity $C(\delta,K)$ from \eqref{eq:CdeltaK}. Since
$C(\delta,K)\ge1$ for every $\delta>0$ and every $K\in\R^{m\times n}$,
we have $\Chat(\delta)\ge1$. Moreover, if $0<\delta_1\le\delta_2$, then
$C(\delta_1,K)\le C(\delta_2,K)$ for every $K\in\R^{m\times n}$, and hence
$\Chat(\delta_1)\le\Chat(\delta_2)$. Thus $\Chat$ is nondecreasing on
$(0,+\infty)$.

For each $\delta>0$, define
\begin{equation}\label{eq:Fdelta}
    \F_\delta
    :=
    \{K\in\R^{m\times n}:C(\delta,K)<+\infty\}.
\end{equation}
Thus $C(\delta,K)$ is finite precisely for $K\in\F_\delta$, while
$C(\delta,K)=+\infty$ for $K\notin\F_\delta$. By
\eqref{eq:Chat-via-CdeltaK}, $\Chat(\delta)=\inf_{K\in\F_\delta}C(\delta,K)$ when
$\F_\delta\neq\varnothing$, whereas $\Chat(\delta)=+\infty$ if
$\F_\delta=\varnothing$. Define the right endpoint of the achievable
decay-rate range by
\begin{equation}\label{eq:deltamax}
\delta_{\max}
:=
\sup\{\delta>0:\F_\delta\neq\varnothing\}
=
\sup\{\delta>0:\Chat(\delta)<+\infty\}.
\end{equation}
In the stabilizable but uncontrollable case, Section~\ref{sec:finite} shows
that $\delta_{\max}=\delta_*$.

For the remainder of this section, assume that $(A,B)$ is controllable. We
record the properties of the controllability indices and the Brunovsk\'y
reduction that will be used in the proof of
Theorem~\ref{thm:main-asymptotic}.

\subsection{Controllability indices and Brunovsk\'y reduction}
\label{sec:controllability-brunovsky}
\label{sec:controllability-indices}
\label{sec:brunovsky}

For $j\ge1$, set
\begin{equation}\label{eq:Rj}
    \mathcal R_j
    :=
    \operatorname{Ran}[B,AB,\ldots,A^{j-1}B],
    \qquad
    \mathcal R_0:=\{0\},
\end{equation}
and write
\begin{equation}\label{8-8-28-w}
    d_j:=\dim\mathcal R_j,
    \qquad
    \rho_j:=d_j-d_{j-1},
    \qquad j\ge1.
\end{equation}
Let
\begin{equation}\label{9-8-28-w}
    r:=\rho_1=\rank B,
\end{equation}
and, for $i=1,\ldots,r$, define
\begin{equation}\label{eq:nu-definition}
    \nu_i
    :=
    \max\{j\ge1:\rho_j\ge i\}.
\end{equation}
The integers $\nu_1,\ldots,\nu_r$ are the positive
\emph{controllability indices} of $(A,B)$, ordered so that
$\nu_1\ge\cdots\ge\nu_r\ge1$.

We will use the following properties. The first is standard; see
\cite{Brunovsky1970,WonhamMorse1972}.

\begin{property}\label{lem:rho-monotone}\label{prop:controllability-indices}
The sequence $(\rho_j)_{j\ge1}$ is nonincreasing,
$\rho_1=\rank B$, and $\rho_j=0$ for all sufficiently large $j$. Moreover,
for every $j\ge1$,
\begin{equation}\label{eq:rho-nu}
    \rho_j
    =
    \#\{i\in\{1,\ldots,r\}:\nu_i\ge j\},
\end{equation}
and
\begin{equation}\label{eq:dj-nu}
    d_j
    =
    \sum_{i=1}^r\min\{j,\nu_i\}.
\end{equation}
In particular,
\begin{equation}\label{eq:sum-nu}
    \nu_1+\cdots+\nu_r=n,
\end{equation}
and the integer $\mu$ defined in \eqref{eq:mu-rank} satisfies
\begin{equation}\label{eq:mu-nu1}
    \mu=\nu_1.
\end{equation}
\end{property}

Associated with these controllability indices, consider the chain system
\begin{equation}\label{eq:Br-system}
\begin{cases}
    \dot z_{i,1}=v_i,\\
    \dot z_{i,j+1}=z_{i,j},
    \qquad j=1,\ldots,\nu_i-1,
\end{cases}
\qquad
i=1,\ldots,r,
\qquad t\ge0.
\end{equation}
For each $i$, the variables
$z_{i,1},\ldots,z_{i,\nu_i}$ and the scalar input $v_i$ form a
controllability chain of length $\nu_i$. By
\eqref{eq:sum-nu}, the state has dimension $n$, and by
\eqref{eq:mu-nu1}, a longest chain has length $\mu=\nu_1$.

For $k\ge1$, let
\[
J_k:=
\begin{pmatrix}
0&0&\cdots&0\\
1&0&\cdots&0\\
0&1&\ddots&\vdots\\
\vdots&\ddots&\ddots&0\\
0&\cdots&1&0
\end{pmatrix}
\in\R^{k\times k},
\qquad
b_k:=
\begin{pmatrix}
1\\
0\\
\vdots\\
0
\end{pmatrix}
\in\R^k.
\]
Then \eqref{eq:Br-system} has the matrix form
\begin{equation}\label{15-8-26}
    \dot z=A_{\rm Br}z+B_{\rm Br}v,
    \qquad t\ge0,
\end{equation}
where
\[
    A_{\rm Br}
    =
    \operatorname{diag}(J_{\nu_1},\ldots,J_{\nu_r}),
    \qquad
    B_{\rm Br}
    =
    \begin{pmatrix}
        b_{\nu_1} & 0 & \cdots & 0\\
        0 & b_{\nu_2} & \cdots & 0\\
        \vdots & \vdots & \ddots & \vdots\\
        0 & 0 & \cdots & b_{\nu_r}
    \end{pmatrix}.
\]
For this pair, define
\[
    C_{\rm Br}(\delta,\widetilde K)
    :=
    \sup_{t\ge0}e^{\delta t}
    \norm{e^{(A_{\rm Br}+B_{\rm Br}\widetilde K)t}},
    \qquad
    \F_\delta^{\rm Br}
    :=
    \{\widetilde K\in\R^{r\times n}:
    C_{\rm Br}(\delta,\widetilde K)<+\infty\},
\]
and
\begin{equation}\label{eq:Chat-Br}
    \Chat_{\rm Br}(\delta)
    :=
    \inf_{\widetilde K\in\R^{r\times n}}
    C_{\rm Br}(\delta,\widetilde K),
    \qquad \delta>0.
\end{equation}

\begin{property}\label{lem:Br-controllable}\label{lem:Br-chain-feedback}
The pair $(A_{\rm Br},B_{\rm Br})$ is controllable. Moreover, under any
state feedback $v=\widetilde Kz$, if $\widetilde K_i$ is the $i$-th row of
$\widetilde K$, then
\begin{equation}\label{17-8-28}
    \dot z=(A_{\rm Br}+B_{\rm Br}\widetilde K)z
\end{equation}
has the chain form
\begin{equation}\label{18-8-28}
\begin{cases}
    \dot z_{i,1}=\widetilde K_i z,\\
    \dot z_{i,j+1}=z_{i,j},
    \qquad j=1,\ldots,\nu_i-1,
\end{cases}
\qquad i=1,\ldots,r.
\end{equation}
In particular, the equations
$\dot z_{i,j+1}=z_{i,j}$ are unaffected by the feedback.
\end{property}

The controllability of $(A_{\rm Br},B_{\rm Br})$ is standard; see
\cite{Brunovsky1970}. The chain equations in
Property~\ref{lem:Br-chain-feedback} follow directly from
\eqref{eq:Br-system}.

\begin{property}\label{lem:remove-redundant-inputs}
Let $B_0\in\R^{n\times r}$ have full column rank and satisfy
$\operatorname{Ran}B_0=\operatorname{Ran}B$. Then
\begin{equation}\label{eq:closed-loop-family-B0}
    \{BK:K\in\R^{m\times n}\}
    =
    \{B_0L:L\in\R^{r\times n}\}.
\end{equation}
Replacing $B$ by $B_0$ leaves $\Chat(\delta)$ unchanged for every
$\delta>0$, and also leaves the spaces $\mathcal R_j$, the controllability
indices, and $\mu$ unchanged. In particular, $(A,B_0)$ is controllable.
\end{property}

A proof of Property~\ref{lem:remove-redundant-inputs} is given in
Appendix~\ref{app:properties}. We may therefore replace $B$ by a
full-column-rank matrix without changing the quantities relevant to
Theorem~\ref{thm:main-asymptotic}. We henceforth denote this matrix again by
$B$ and assume that
\begin{equation}\label{21-8-28}
B\in\R^{n\times r},
\qquad
\rank B=r.
\end{equation}

The following classical theorem provides the structural reduction used
below; see \cite{Brunovsky1970}.

\begin{theorem}[Brunovsk\'y canonical form]
\label{thm:Brunovsky-canonical-form}
Assume that $(A,B)$ is controllable and that $B$ satisfies
\eqref{21-8-28}. Let $A_{\rm Br}$ and $B_{\rm Br}$ be the matrices in
\eqref{15-8-26} associated with the controllability indices of $(A,B)$.
Then there exist invertible matrices
$T\in\R^{n\times n}$ and $U\in\R^{r\times r}$, and a matrix
$K_{\rm pre}\in\R^{r\times n}$, such that
\begin{equation}\label{eq:Br-feedback-equivalence}
T^{-1}(A+BK_{\rm pre})T=A_{\rm Br},
\qquad
T^{-1}BU=B_{\rm Br}.
\end{equation}
Equivalently, under the state and input-feedback transformation
$x=Tz$ and $u=K_{\rm pre}x+Uv$, the system
$\dot x=Ax+Bu$ becomes \eqref{15-8-26}.
\end{theorem}

Once a Brunovsk\'y transformation is fixed, it induces a one-to-one
correspondence between the feedback gains of the two systems.

\begin{property}\label{prop:feedback-decay-correspondence}
Let $T$, $U$, and $K_{\rm pre}$ be as in
Theorem~\ref{thm:Brunovsky-canonical-form}. The map
\begin{equation}\label{eq:feedback-correspondence}
    G:\R^{r\times n}\to\R^{r\times n},
    \qquad
    G(K):=U^{-1}(K-K_{\rm pre})T,
\end{equation}
is a bijection, with inverse
\begin{equation}\label{eq:feedback-correspondence-inverse}
    G^{-1}(\widetilde K)
    =
    K_{\rm pre}+U\widetilde K T^{-1}.
\end{equation}
Moreover,
\begin{equation}\label{eq:closed-loop-similarity}
    A_{\rm Br}+B_{\rm Br}G(K)
    =
    T^{-1}(A+BK)T,
    \qquad K\in\R^{r\times n}.
\end{equation}
\end{property}

\begin{property}\label{prop:Br-reduction}
There exists a constant $\kappa\ge1$, independent of $\delta$, such that
\begin{equation}\label{eq:Chat-comparison}
    \frac1{\kappa}\Chat_{\rm Br}(\delta)
    \le
    \Chat(\delta)
    \le
    \kappa\Chat_{\rm Br}(\delta),
    \qquad \delta>0.
\end{equation}
\end{property}

Proofs of Properties~\ref{prop:feedback-decay-correspondence} and
\ref{prop:Br-reduction} are given in Appendix~\ref{app:properties}.

\begin{remark}\label{rem:Br-reduction-meaning}
The state transformation $T$ need not be orthogonal, so the Euclidean
operator norm is not preserved and $\Chat(\delta)$ and
$\Chat_{\rm Br}(\delta)$ need not be equal. However,
\eqref{eq:Chat-comparison} shows that they are comparable with constants
independent of $\delta$. Consequently, they have the same endpoint growth
exponent.
\end{remark}

\section{Proof of Theorem~\ref{thm:main-asymptotic}}
\label{sec:controllable}

We first establish the lower and upper bounds for a controllable pair
$(A,B)$ satisfying \eqref{21-8-28}. These bounds will then be applied
to the full-column-rank pair obtained from the original pair in the proof
of Theorem~\ref{thm:main-asymptotic}.

The following proposition will be used in the proof of
Theorem~\ref{thm:main-asymptotic}.

\begin{proposition}[Lower bound in Brunovsk\'y form]
\label{prop:canonical-lower}
Assume that $(A,B)$ is controllable and that $B$ satisfies
\eqref{21-8-28}. Then 
\begin{equation}\label{eq:canonical-lower}
\Chat_{\rm Br}(\delta)
\ge
\delta^{\mu-1},  \qquad \delta>0.
\end{equation}
\end{proposition}

\begin{proof}
We fix an arbitrary $\delta>0$. It suffices to prove that, for every
$\widetilde K\in\R^{r\times n}$,
\begin{equation}\label{32-8-28-wg}
C_{\rm Br}(\delta,\widetilde K)
\ge
\delta^{\mu-1},
\qquad \delta>0,\quad \widetilde K\in\R^{r\times n}.
\end{equation}

We first consider the case $\mu=1$. By evaluating the expression defining
$C_{\rm Br}(\delta,\widetilde K)$ at $t=0$, we obtain
\eqref{32-8-28-wg}.

We next consider the case $\mu\ge2$. If
$\widetilde K\notin\F_\delta^{\rm Br}$, then
$C_{\rm Br}(\delta,\widetilde K)=+\infty$, and hence
\eqref{32-8-28-wg} holds.

Now let $\widetilde K\in\F_\delta^{\rm Br}$ and set
$M:=A_{\rm Br}+B_{\rm Br}\widetilde K$. Then $M$ is Hurwitz and hence
invertible. By the definition of
$C_{\rm Br}(\delta,\widetilde K)$,
\begin{equation}\label{eq:M-decay-lower}
\norm{e^{Mt}}
\le
C_{\rm Br}(\delta,\widetilde K)e^{-\delta t},
\qquad t\ge0,\quad \delta>0,\quad \widetilde K\in\F_\delta^{\rm Br}.
\end{equation}

Let $e_1,\ldots,e_n$ be the standard basis of $\R^n$. Since $(A,B)$ is
controllable, Property~\ref{prop:controllability-indices} applies and
\eqref{eq:mu-nu1} shows that the first chain has length $\mu$. Moreover,
Property~\ref{lem:Br-chain-feedback} applies. Since the first chain occupies
the first $\mu$ coordinates, that lemma gives
$e_{j+1}^\top M=e_j^\top$ for $j=1,\ldots,\mu-1$. Since $M$ is
invertible, $e_j^\top M^{-1}=e_{j+1}^\top$, and iteration yields
$e_1^\top M^{-(\mu-1)}=e_\mu^\top$. Consequently,
\begin{equation}\label{33-8-28}
1\le\norm{M^{-(\mu-1)}}.
\end{equation}

On the other hand, by the Laplace-transform formula for negative powers of
a Hurwitz matrix, we have
\begin{equation}\label{eq:negative-power-laplace}
    M^{-(\mu-1)}
    =
    \frac{(-1)^{\mu-1}}{(\mu-2)!}
    \int_0^\infty t^{\mu-2}e^{Mt}\,dt.
\end{equation}
By \eqref{eq:negative-power-laplace}, \eqref{eq:M-decay-lower}, and
$\int_0^\infty t^{\mu-2}e^{-\delta t}\,dt
=(\mu-2)!/\delta^{\mu-1}$ for $\delta>0$ and $\mu\ge2$, we obtain
\begin{equation}\label{34-8-28}
\norm{M^{-(\mu-1)}}
\le
\frac{C_{\rm Br}(\delta,\widetilde K)}
{\delta^{\mu-1}},
\qquad \delta>0,\quad \widetilde K\in\F_\delta^{\rm Br}.
\end{equation}
By \eqref{33-8-28} and \eqref{34-8-28}, we obtain
\eqref{32-8-28-wg} when $\mu\ge2$ and
$\widetilde K\in\F_\delta^{\rm Br}$.

Thus \eqref{32-8-28-wg} holds for every
$\widetilde K\in\R^{r\times n}$. Taking the infimum over
$\widetilde K\in\R^{r\times n}$ proves \eqref{eq:canonical-lower}.
\end{proof}

The following proposition will be used in the proof of
Theorem~\ref{thm:main-asymptotic}.

\begin{proposition}[Upper bound in Brunovsk\'y form]
\label{prop:canonical-upper}
Assume that $(A,B)$ is controllable and that $B$ satisfies
\eqref{21-8-28}. Then there exist constants $C_{\rm up}>0$ and $\delta_{\rm up}>0$ such that
\begin{equation}\label{eq:canonical-upper}
    \Chat_{\rm Br}(\delta)
    \le
    C_{\rm up}\delta^{\mu-1},
    \qquad
    \delta\ge\delta_{\rm up}.
\end{equation}
\end{proposition}
\begin{proof}
Since $(A,B)$ is controllable, Property~\ref{lem:Br-controllable} applies.
Hence we may choose $\widetilde K_0\in\R^{r\times n}$ such that
$H:=A_{\rm Br}+B_{\rm Br}\widetilde K_0$ is Hurwitz. Then there exist
constants $C_H\ge1$ and $\eta>0$ such that
\begin{equation}\label{eq:H-decay}
\norm{e^{Ht}}
\le
C_He^{-\eta t},
\qquad t\ge0.
\end{equation}

For $s\ge1$ and $i=1,\ldots,r$, define
\[
    D_{i,s}
    :=
    \operatorname{diag}
    (s^{\nu_i-1},s^{\nu_i-2},\ldots,s,1),
    \qquad s\ge1,\quad i=1,\ldots,r.
\]
with $D_{i,s}=(1)$ when $\nu_i=1$, and set
$D_s:=\operatorname{diag}(D_{1,s},\ldots,D_{r,s})\in\R^{n\times n}$.
Since $(A,B)$ is controllable,
Property~\ref{prop:controllability-indices} applies and gives
$\mu=\max_{1\le i\le r}\nu_i$. Since $D_s$ is diagonal and $s\ge1$,
the Euclidean operator norm gives
\begin{equation}\label{eq:D-condition}
    \norm{D_s}=s^{\mu-1},
    \qquad
    \norm{D_s^{-1}}=1,
    \qquad s\ge1.
\end{equation}
Set $R_s:=\operatorname{diag}(s^{\nu_1},\ldots,s^{\nu_r})\in\R^{r\times r}$.
By the definitions of $J_{\nu_i}$, $b_{\nu_i}$, and $D_{i,s}$, a direct
computation shows
\begin{equation}\label{eq:Br-scaling-identities}
    D_s^{-1}A_{\rm Br}D_s=sA_{\rm Br},
    \qquad
    D_s^{-1}B_{\rm Br}R_s=sB_{\rm Br},
    \qquad s\ge1.
\end{equation}
Now define $\widetilde K_s:=R_s\widetilde K_0D_s^{-1}$. By
\eqref{eq:Br-scaling-identities}, we obtain
\begin{equation}\label{eq:high-gain-similarity}
    D_s^{-1}
    (A_{\rm Br}+B_{\rm Br}\widetilde K_s)
    D_s
    =
    sH,
    \qquad s\ge1.
\end{equation}
By \eqref{eq:high-gain-similarity},
$$
    e^{(A_{\rm Br}+B_{\rm Br}\widetilde K_s)t}
    =
    D_se^{sHt}D_s^{-1},
    \qquad t\ge0,\quad s\ge1.
$$
Together with \eqref{eq:H-decay} and \eqref{eq:D-condition}, this yields
\begin{equation}\label{eq:scaled-decay}
\norm{e^{(A_{\rm Br}+B_{\rm Br}\widetilde K_s)t}}
\le
C_Hs^{\mu-1}e^{-\eta st},
\qquad t\ge0,\quad s\ge1.
\end{equation}

For $\delta\ge\eta$, choose $s=\delta/\eta\ge1$. Then
\eqref{eq:scaled-decay} gives

$$
    \norm{e^{(A_{\rm Br}+B_{\rm Br}\widetilde K_s)t}}
    \le
    C_H\eta^{-(\mu-1)}\delta^{\mu-1}e^{-\delta t},
    \qquad t\ge0,\quad \delta\ge\eta.
$$

By the definition of $C_{\rm Br}(\delta,\widetilde K_s)$, we obtain
$$
    C_{\rm Br}(\delta,\widetilde K_s)
    \le
    C_H\eta^{-(\mu-1)}\delta^{\mu-1},
    \qquad \delta\ge\eta.
$$
Hence, by the definition of $\Chat_{\rm Br}(\delta)$,
$$
    \Chat_{\rm Br}(\delta)
    \le
    C_H\eta^{-(\mu-1)}\delta^{\mu-1},
    \qquad
    \delta\ge\eta.
$$
Thus \eqref{eq:canonical-upper} holds with
$C_{\rm up}=C_H\eta^{-(\mu-1)}$ and $\delta_{\rm up}=\eta$.
\end{proof}

\vskip 5pt
We now return to the original controllable pair and prove
Theorem~\ref{thm:main-asymptotic}.

\begin{proof}[Proof of Theorem~\ref{thm:main-asymptotic}]
Let $r=\rank B$, and choose $B_0\in\R^{n\times r}$ with full column rank
such that $\operatorname{Ran}B_0=\operatorname{Ran}B$. Since $(A,B)$ is
controllable, Property~\ref{lem:remove-redundant-inputs} applies. Hence
$(A,B_0)$ is controllable, and replacing $B$ by $B_0$ leaves
$\Chat(\delta)$, the controllability indices, and $\mu$ unchanged.

First, since $(A,B)$ and $(A,B_0)$ have the same controllability indices, the
Brunovsk\'y canonical pair associated with $(A,B_0)$ is exactly the pair
$(A_{\rm Br},B_{\rm Br})$ defined in \eqref{15-8-26}.
Second, since $(A,B_0)$ is
controllable and $B_0$ has full column rank,
Property~\ref{prop:Br-reduction} applies to $(A,B_0)$.
Third, by
Property~\ref{lem:remove-redundant-inputs}, the optimal stabilization
prefactor of $(A,B_0)$
is the same as that of $(A,B)$. Combining these three facts with
\eqref{eq:Chat-comparison} yields
\begin{equation}\label{eq:main-Br-comparison}
    \frac1{\kappa}\Chat_{\rm Br}(\delta)
    \le
    \Chat(\delta)
    \le
    \kappa\Chat_{\rm Br}(\delta),
    \qquad \delta>0,
\end{equation}
where $\kappa\ge1$ is independent of $\delta$.

We first prove the lower bound. Since $(A,B_0)$ is controllable and $B_0$
has full column rank, Proposition~\ref{prop:canonical-lower} applies to
$(A,B_0)$. By Property~\ref{lem:remove-redundant-inputs}, the integer $\mu$
for $(A,B_0)$ is the same as that for $(A,B)$, and by the first fact above,
its Brunovsk\'y canonical pair is $(A_{\rm Br},B_{\rm Br})$. Hence
\eqref{eq:canonical-lower} gives
\begin{equation}\label{eq:main-Br-lower}
    \Chat_{\rm Br}(\delta)
    \ge
    \delta^{\mu-1},
    \qquad \delta>0.
\end{equation}
By \eqref{eq:main-Br-comparison} and \eqref{eq:main-Br-lower}, we obtain
\begin{equation}\label{eq:main-lower-final}
    \Chat(\delta)
    \ge
    \frac1{\kappa}\delta^{\mu-1},
    \qquad \delta>0.
\end{equation}

We next prove the upper bound. Since $(A,B_0)$ is controllable and $B_0$
has full column rank, Proposition~\ref{prop:canonical-upper} applies to
$(A,B_0)$. By Property~\ref{lem:remove-redundant-inputs}, the integer $\mu$
for $(A,B_0)$ is the same as that for $(A,B)$, and by the first fact above,
its Brunovsk\'y canonical pair is $(A_{\rm Br},B_{\rm Br})$. Hence
\eqref{eq:canonical-upper} gives constants $C_{\rm up}>0$ and
$\delta_{\rm up}>0$ such that
\begin{equation}\label{eq:main-Br-upper}
    \Chat_{\rm Br}(\delta)
    \le
    C_{\rm up}\delta^{\mu-1},
    \qquad
    \delta\ge\delta_{\rm up}.
\end{equation}
By \eqref{eq:main-Br-comparison} and \eqref{eq:main-Br-upper}, we obtain
\begin{equation}\label{eq:main-upper-final}
    \Chat(\delta)
    \le
    \kappa C_{\rm up}\delta^{\mu-1},
    \qquad
    \delta\ge\delta_{\rm up}.
\end{equation}

By \eqref{eq:main-lower-final} and \eqref{eq:main-upper-final},
\eqref{eq:main-asymptotic} holds with
$c_0=\kappa^{-1}$, $C_0=\kappa C_{\rm up}$, and
$\delta_0=\delta_{\rm up}$.

Taking logarithms in \eqref{eq:main-asymptotic}, dividing by
$\log\delta$, and letting $\delta\to+\infty$, we obtain
\eqref{eq:log-slope}.

Finally, suppose that $\mu=1$. Then $\rank B=n$. Choose
$R\in\R^{m\times n}$ such that $BR=I_n$ and set
$K_\delta:=-R(A+\delta I_n)$. Then
$A+BK_\delta=-\delta I_n$, so $C(\delta,K_\delta)=1$. Since evaluating at
$t=0$ gives $C(\delta,K)\ge1$ for every $K$, we obtain
\eqref{eq:mu-one-exact}.

\end{proof}

\section{Proof of Theorem~\ref{thm:finite-endpoint}}
\label{sec:finite}

Assume throughout this section that $(A,B)$ is stabilizable but not
controllable. Let
\begin{equation}\label{eq:reachable-space}
    \mathcal R
    :=
    \operatorname{span}\{\Ran B,\Ran AB,\ldots,\Ran A^{n-1}B\}
\end{equation}
be the controllable subspace, and set
$n_{\rm c}:=\dim\mathcal R$ and $n_{\rm u}:=n-n_{\rm c}\ge1$.
Choose an orthogonal matrix $Q=(Q_{\rm c},Q_{\rm u})$, where
$Q_{\rm c}\in\R^{n\times n_{\rm c}}$ and
$Q_{\rm u}\in\R^{n\times n_{\rm u}}$, whose columns form orthonormal
bases of $\mathcal R$ and $\mathcal R^\perp$, respectively. Since
$\mathcal R$ is $A$-invariant and $\Ran B\subset\mathcal R$,
\begin{equation}\label{eq:orthogonal-Kalman}
    Q^\top A Q
    =
    \begin{pmatrix}
        A_{\rm c}&A_{\rm cu}\\
        0&A_{\rm u}
    \end{pmatrix},
    \qquad
    Q^\top B
    =
    \begin{pmatrix}
        B_{\rm c}\\
        0
    \end{pmatrix},
\end{equation}
where
\[
    A_{\rm c}\in\R^{n_{\rm c}\times n_{\rm c}},
    \qquad
    A_{\rm cu}\in\R^{n_{\rm c}\times n_{\rm u}},
    \qquad
    A_{\rm u}\in\R^{n_{\rm u}\times n_{\rm u}},
    \qquad
    B_{\rm c}\in\R^{n_{\rm c}\times m},
\]
and $(A_{\rm c},B_{\rm c})$ is controllable.

For $K\in\R^{m\times n}$, write
\[
    KQ=(K_{\rm c},K_{\rm u}),
    \qquad
    K\in\R^{m\times n},
    \qquad
    K_{\rm c}\in\R^{m\times n_{\rm c}},
    \qquad
    K_{\rm u}\in\R^{m\times n_{\rm u}}.
\]
Then, by \eqref{eq:orthogonal-Kalman},
\begin{equation}\label{eq:closed-loop-block}
    Q^\top(A+BK)Q
    =
    \begin{pmatrix}
        A_{\rm c}+B_{\rm c}K_{\rm c}
        &A_{\rm cu}+B_{\rm c}K_{\rm u}\\
        0&A_{\rm u}
    \end{pmatrix},
    \qquad K\in\R^{m\times n}.
\end{equation}
In particular, the uncontrollable block $A_{\rm u}$ is unaffected by
the feedback $K$.

The following standard proposition will be used below to identify
$\delta_*$ and $q$, and in the proofs of
Proposition~\ref{prop:maximal-rate} and
Theorem~\ref{thm:finite-endpoint}. For its proof, see
\cite{Kalman1963,Hautus1970}.

\begin{proposition}\label{prop:uncontrollable-block}
Assume that $(A,B)$ is not controllable. The following statements hold.
\begin{enumerate}[label=\textup{(\roman*)}]
    \item The pair $(A,B)$ is stabilizable if and only if $A_{\rm u}$ is
    Hurwitz.
    \item If another orthogonal Kalman decomposition is used, then its
    uncontrollable block is orthogonally similar to $A_{\rm u}$. In
    particular, the spectrum and Jordan block sizes of the uncontrollable
    block do not depend on the choice of orthogonal Kalman decomposition.
\end{enumerate}
\end{proposition}

Since $(A,B)$ is stabilizable but not controllable,
Proposition~\ref{prop:uncontrollable-block}\textup{(i)} applies and shows
that $A_{\rm u}$ is Hurwitz. Let
\begin{equation}\label{eq:su}
    s_{\rm u}
    :=
    \max\{\operatorname{Re}\lambda:\lambda\in\sigma(A_{\rm u})\}<0
\end{equation}
be its spectral abscissa. By \eqref{eq:deltastar},
$\delta_*=-s_{\rm u}>0$. Recall that $q$ is the largest size of a Jordan
block in the complex Jordan form of $A_{\rm u}$ associated with an
eigenvalue $\lambda$ satisfying
\begin{equation}\label{eq:peripheral-line}
    \operatorname{Re}\lambda=s_{\rm u}=-\delta_*,
    \qquad \lambda\in\sigma(A_{\rm u}).
\end{equation}
Since $(A,B)$ is stabilizable but not controllable,
Proposition~\ref{prop:uncontrollable-block}\textup{(ii)} applies. Hence
both $\delta_*$ and $q$ are independent of the choice of orthogonal Kalman
decomposition.

For later use, note that the lower-right block in
\eqref{eq:closed-loop-block} is always $A_{\rm u}$. Hence, for every
$K\in\R^{m\times n}$,
\[
    e^{Q^\top(A+BK)Qt}
    =
    \begin{pmatrix}
        *&*\\
        0&e^{A_{\rm u}t}
    \end{pmatrix},
    \qquad K\in\R^{m\times n},\quad t\ge0.
\]
It follows that
\[
    \norm{e^{Q^\top(A+BK)Qt}}
    \ge
    \norm{e^{A_{\rm u}t}},
    \qquad K\in\R^{m\times n},\quad t\ge0.
\]
Since $Q$ is orthogonal,
\[
    \norm{e^{Q^\top(A+BK)Qt}}
    =
    \norm{e^{(A+BK)t}},
    \qquad K\in\R^{m\times n},\quad t\ge0.
\]
Therefore
\begin{equation}\label{eq:block-lower-bound}
    \norm{e^{(A+BK)t}}
    \ge
    \norm{e^{A_{\rm u}t}},
    \qquad K\in\R^{m\times n},\quad t\ge0.
\end{equation}

Set $H_{\rm u}:=A_{\rm u}+\delta_*I_{n_{\rm u}}$. The following lemma
will be used in the proof of Proposition~\ref{prop:maximal-rate} and
Theorem~\ref{thm:finite-endpoint}. Its proof, based on the complex Jordan
normal form, is included in Appendix~\ref{app:kalman}.

\begin{lemma}\label{lem:Hu-growth}
Assume that $(A,B)$ is stabilizable but not controllable. There exist constants $c_1,C_1>0$ and $T_1\ge1$ such that
\begin{equation}\label{eq:Hu-growth}
    c_1t^{q-1}
    \le
    \norm{e^{H_{\rm u}t}}
    \le
    C_1t^{q-1},
    \qquad t\ge T_1,
\end{equation}
and, after increasing $C_1$ if necessary,
\begin{equation}\label{eq:Hu-global-growth}
    \norm{e^{H_{\rm u}t}}
    \le
    C_1(1+t^{q-1}),
    \qquad t\ge0.
\end{equation}
Moreover, there exist constants $c_2,C_2,\varepsilon_1>0$ such that
\begin{equation}\label{eq:weighted-Hu-growth}
    c_2\varepsilon^{-(q-1)}
    \le
    \sup_{t\ge0}e^{-\varepsilon t}\norm{e^{H_{\rm u}t}}
    \le
    C_2\varepsilon^{-(q-1)},
    \qquad 0<\varepsilon\le\varepsilon_1.
\end{equation}
For $q=1$, the factor $\varepsilon^{-(q-1)}$ is understood as $1$.
\end{lemma}

The following proposition will be used in the proof of
Theorem~\ref{thm:finite-endpoint}. For completeness, its proof is included
in Appendix~\ref{app:kalman}.

\begin{proposition}\label{prop:maximal-rate}
Assume that $(A,B)$ is stabilizable but not controllable. The following statements hold.
\begin{enumerate}[label=\textup{(\roman*)}]
    \item $\F_\delta\neq\varnothing$ for every $0<\delta<\delta_*$;
    \item $\F_\delta=\varnothing$ for every $\delta>\delta_*$;
    \item $\F_{\delta_*}\neq\varnothing$ if and only if $q=1$.
\end{enumerate}
Moreover,
\begin{equation}\label{eq:maximal-rate}
    \delta_*
    =
    \sup\{\delta>0:\F_\delta\neq\varnothing\}.
\end{equation}
Thus $\delta_{\max}=\delta_*$ in the notation of \eqref{eq:deltamax}.
\end{proposition}

\vskip 5pt
We now prove Theorem~\ref{thm:finite-endpoint}.

\begin{proof}
Since $(A,B)$ is stabilizable but not controllable,
Proposition~\ref{prop:maximal-rate} applies and gives
$\delta_{\max}=\delta_*$. Moreover,
Proposition~\ref{prop:uncontrollable-block}\textup{(ii)} applies. Hence
both $\delta_*$ and $q$ are independent of the choice of orthogonal Kalman
decomposition. Throughout the proof, we use the orthogonal Kalman
decomposition fixed in \eqref{eq:orthogonal-Kalman}.

\medskip
\noindent\textit{Lower bound.}
Fix $\delta\in(0,\delta_*)$ and set
$\varepsilon:=\delta_*-\delta>0$. Let
$K\in\R^{m\times n}$ be arbitrary. By
\eqref{eq:block-lower-bound} and the identity
$A_{\rm u}=H_{\rm u}-\delta_*I_{n_{\rm u}}$,  we obtain
\begin{equation}\label{50-8-31}
    e^{\delta t}\norm{e^{(A+BK)t}}
    \ge
    e^{-\varepsilon t}\norm{e^{H_{\rm u}t}},\;\;t\ge0.
\end{equation}
Taking the supremum over $t\ge0$ in \eqref{50-8-31} and using the
definition of $C(\delta,K)$, we obtain
\begin{equation}\label{51-8-31}
    C(\delta,K)
    \ge
    \sup_{t\ge0}
    e^{-\varepsilon t}\norm{e^{H_{\rm u}t}}.
\end{equation}
Since $K\in\R^{m\times n}$ was arbitrary and the right-hand side of
\eqref{51-8-31} is independent of $K$, taking the infimum over
$K\in\R^{m\times n}$ in \eqref{51-8-31} and using the definition of
$\Chat(\delta)$, we obtain
\begin{equation}\label{54-8-31}
    \Chat(\delta)
    \ge
    \sup_{t\ge0}
    e^{-\varepsilon t}\norm{e^{H_{\rm u}t}}.
\end{equation}

We now estimate the right-hand side of \eqref{54-8-31}.
Since $(A,B)$ is stabilizable but not controllable,
Lemma~\ref{lem:Hu-growth} applies. Hence there exists
$\varepsilon_1>0$ such that, for every
$\varepsilon'\in(0,\varepsilon_1]$,
\begin{equation}\label{52-8-31}
    \sup_{t\ge0}
    e^{-\varepsilon' t}\norm{e^{H_{\rm u}t}}
    \ge
    c_2(\varepsilon')^{-(q-1)},
    \qquad 0<\varepsilon'\le\varepsilon_1.
\end{equation}
Since $\varepsilon=\delta_*-\delta$ tends to $0$ as
$\delta\uparrow\delta_*$, it follows that when $\delta$ is sufficiently
close to $\delta_*$, $0<\varepsilon\le\varepsilon_1$. Taking
$\varepsilon'=\varepsilon$ in \eqref{52-8-31}, we obtain
$$
    \sup_{t\ge0}
    e^{-\varepsilon t}\norm{e^{H_{\rm u}t}}
    \ge
    c_2\varepsilon^{-(q-1)},
    \qquad 0<\varepsilon\le\varepsilon_1.
$$
This, along with \eqref{54-8-31}, yields
\begin{equation}\label{eq:finite-main-lower}
    \Chat(\delta)
    \ge
    c_2\varepsilon^{-(q-1)},
    \qquad
    \delta\in(0,\delta_*),\quad \varepsilon=\delta_*-\delta,
    \quad 0<\varepsilon\le\varepsilon_1.
\end{equation}

\medskip
\noindent\textit{Upper bound.}
If $n_{\rm c}=0$, then $B=0$ and $A_{\rm u}=A$, so
$\Chat(\delta)=\sup_{t\ge0}e^{-\varepsilon t}\norm{e^{H_{\rm u}t}}$.
Thus \eqref{eq:weighted-Hu-growth} gives the required upper bound directly.
Together with \eqref{eq:finite-main-lower}, this proves
\eqref{eq:finite-endpoint-asymptotic} when $n_{\rm c}=0$, and
\eqref{eq:recover-q} follows by the same logarithmic argument used below.
Hence we may assume $n_{\rm c}\ge1$. Fix
$\delta\in(0,\delta_*)$ and set
$\varepsilon:=\delta_*-\delta>0$. Fix $\eta:=1$. Since $(A_{\rm c},B_{\rm c})$ is controllable, by pole
assignment we may choose $K_{\rm c}$ such that, with
$M_{\rm c}:=A_{\rm c}+B_{\rm c}K_{\rm c}$,

$$
    \sigma(M_{\rm c})
    \subset
    \{z\in\C:\operatorname{Re}z<-\delta_*-\eta\}.
$$

Then there exists $M_{\rm c}^{(0)}>0$ such that
\begin{equation}\label{eq:Gc-decay}
\norm{e^{G_{\rm c}t}}
\le
M_{\rm c}^{(0)}e^{-\eta t},
\qquad
t\ge0,
\end{equation}
where $G_{\rm c}:=M_{\rm c}+\delta_*I_{n_{\rm c}}$.

Set $K_0:=(K_{\rm c},0)Q^\top$. Then
$K_0Q=(K_{\rm c},0)$. This, together with
\eqref{eq:closed-loop-block}, yields
\begin{equation}\label{eq:finite-M0}
M_0
:=
Q^\top(A+BK_0)Q
=
\begin{pmatrix}
M_{\rm c}&A_{\rm cu}\\
0&A_{\rm u}
\end{pmatrix}.
\end{equation}
By \eqref{eq:finite-M0} and the definitions of $G_{\rm c}$ and
$H_{\rm u}$, we obtain

$$
    M_0+\delta_*I_n
    =
    \begin{pmatrix}
        G_{\rm c}&A_{\rm cu}\\
        0&H_{\rm u}
    \end{pmatrix}.
$$
The standard formula for the exponential of a block upper-triangular
matrix then gives
\begin{equation}\label{eq:block-semigroup}
e^{(M_0+\delta_*I_n)t}
=
\begin{pmatrix}
e^{G_{\rm c}t}&R(t)\\
0&e^{H_{\rm u}t}
\end{pmatrix},
\qquad t\ge0,
\end{equation}
where
$$
    R(t)
    =
    \int_0^t
    e^{G_{\rm c}(t-s)}
    A_{\rm cu}
    e^{H_{\rm u}s}\,ds,
    \qquad t\ge0.
$$
The above expression, along with \eqref{eq:Gc-decay} and
Lemma~\ref{lem:Hu-growth}, in particular \eqref{eq:Hu-global-growth}, implies
$$
\begin{aligned}
    \norm{R(t)}
    &\le
    M_{\rm c}^{(0)}\norm{A_{\rm cu}}C_1
    \int_0^t e^{-\eta(t-s)}(1+s^{q-1})\,ds\\
    &\le
    \frac{M_{\rm c}^{(0)}\norm{A_{\rm cu}}C_1}{\eta}
    (1+t^{q-1}):=C_R(1+t^{q-1}),
    \qquad t\ge0.
\end{aligned}
$$
Together with \eqref{eq:block-semigroup}, \eqref{eq:Gc-decay}, and
\eqref{eq:Hu-global-growth}, this yields a constant $C_M>0$ such that
\begin{equation}\label{eq:shifted-M0-growth}
\norm{e^{(M_0+\delta_*I_n)t}}
\le
C_M(1+t^{q-1}),
\qquad
t\ge0.
\end{equation}

Meanwhile, by the definition of $M_0$ in \eqref{eq:finite-M0} and the orthogonality
of $Q$, we have

$$
    \norm{e^{(A+BK_0)t}}
    =
    \norm{e^{M_0t}},
    \qquad t\ge0.
$$
This, together with  the definition of $C(\delta,K_0)$ and
$\delta=\delta_*-\varepsilon$, gives
$$
    C(\delta,K_0)
  =
    \sup_{t\ge0}
    e^{\delta t}\norm{e^{M_0t}}
    =
    \sup_{t\ge0}
    e^{-\varepsilon t}
    \norm{e^{(M_0+\delta_*I_n)t}}.
$$
Then, by \eqref{eq:shifted-M0-growth}, we obtain
\begin{equation}\label{61-8-29}
    C(\delta,K_0)
    \le
    C_M
    \sup_{t\ge0}
    e^{-\varepsilon t}(1+t^{q-1}),
    \qquad \varepsilon>0.
\end{equation}

Since $\varepsilon\downarrow0$ as $\delta\uparrow\delta_*$, for $\delta$
sufficiently close to $\delta_*$ we also have
$0<\varepsilon\le1$. By the change of variable $\tau=\varepsilon t$,

$$
\begin{aligned}
    \sup_{t\ge0}e^{-\varepsilon t}(1+t^{q-1})
    &\le
    1+
    \varepsilon^{-(q-1)}
    \sup_{\tau\ge0}e^{-\tau}\tau^{q-1}\\
    &\le
    C_*\varepsilon^{-(q-1)},
    \qquad 0<\varepsilon\le1,
\end{aligned}
$$

for some $C_*>0$ independent of $\varepsilon$. Consequently, by the
definition of $\Chat(\delta)$ and \eqref{61-8-29}, we find
\begin{equation}\label{eq:finite-main-upper}
\Chat(\delta)
\le
C(\delta,K_0)
\le
C_MC_*\varepsilon^{-(q-1)},
\qquad
\delta\in(0,\delta_*),\quad \varepsilon=\delta_*-\delta,
\quad 0<\varepsilon\le1.
\end{equation}

Finally, it follows from \eqref{eq:finite-main-lower} and
\eqref{eq:finite-main-upper} that 
\eqref{eq:finite-endpoint-asymptotic} holds with
$c_0=c_2$, $C_0=C_MC_*$, and
$\varepsilon_0=\min\{\varepsilon_1,1,\delta_*/2\}$,
 all depending only on $(A,B)$.
Taking logarithms in the two inequalities in
\eqref{eq:finite-endpoint-asymptotic}, dividing by
$\log(1/(\delta_*-\delta))$, and letting
$\delta\uparrow\delta_*$, we obtain \eqref{eq:recover-q}.

Finally, Proposition~\ref{prop:maximal-rate}\textup{(iii)} gives
$\F_{\delta_*}\neq\varnothing$ if and only if $q=1$. By the definition of
$\F_{\delta_*}$ and \eqref{eq:Chat-via-CdeltaK},
$\F_{\delta_*}\neq\varnothing$ if and only if
$\Chat(\delta_*)<+\infty$. This proves \eqref{eq:q-one-endpoint}.
\end{proof}

\section{Examples}
\label{sec:examples}

The examples below illustrate two distinct features of the main results. In the controllable case, the asymptotic growth of $\Chat(\delta)$ is determined by the largest controllability index rather than by the state dimension. In the stabilizable but uncontrollable case, the behavior near $\delta_*$ is determined by the largest size of a Jordan block associated with the spectral boundary, rather than by the largest Jordan-block size of the uncontrollable part as a whole.

\medskip
\noindent\textbf{Example 1: same dimension, different controllability structures.}
Consider the following three systems in $\R^5$:
$$
A_1=
\begin{pmatrix}
0&0&0&0&0\\
1&0&0&0&0\\
0&1&0&0&0\\
0&0&1&0&0\\
0&0&0&1&0
\end{pmatrix},
\quad
B_1=
\begin{pmatrix}
1\\0\\0\\0\\0
\end{pmatrix},
\qquad
A_2=
\begin{pmatrix}
0&0&0&0&0\\
1&0&0&0&0\\
0&1&0&0&0\\
0&0&0&0&0\\
0&0&0&1&0
\end{pmatrix},
\quad
B_2=
\begin{pmatrix}
1&0\\
0&0\\
0&0\\
0&1\\
0&0
\end{pmatrix},
$$
$$
A_3=
\begin{pmatrix}
0&0&0&0&0\\
1&0&0&0&0\\
0&0&0&0&0\\
0&0&1&0&0\\
0&0&0&0&0
\end{pmatrix},
\qquad
B_3=
\begin{pmatrix}
1&0&0\\
0&0&0\\
0&1&0\\
0&0&0\\
0&0&1
\end{pmatrix}.
$$
By the definition of the controllability indices, we obtain, respectively,
$$
(5),\qquad (3,2),\qquad (2,2,1).
$$
Hence the corresponding largest controllability indices are $5$, $3$, and $2$. By Theorem~\ref{thm:main-asymptotic}, we obtain that the corresponding optimal stabilization prefactors satisfy, respectively,
$$
\Chat(\delta)\asymp\delta^4,\qquad
\Chat(\delta)\asymp\delta^2,\qquad
\Chat(\delta)\asymp\delta,
\qquad \delta\to+\infty.
$$

Thus the three systems have the same state dimension and are all controllable, but the asymptotic growth orders of their optimal stabilization prefactors are different. This shows that the growth order is determined by the largest controllability index rather than by the state dimension.

\medskip
\noindent\textbf{Example 2: two phenomena at a finite endpoint.}
\smallskip
\noindent\textit{(a) A finite endpoint without blow-up.}
Let
$$
A=\operatorname{diag}(0,-1,-2),
\qquad
B=
\begin{pmatrix}
1\\0\\0
\end{pmatrix}.
$$
The pair $(A,B)$ is stabilizable but not controllable, and its uncontrollable block is
$$
A_{\rm u}=\operatorname{diag}(-1,-2).
$$
By \eqref{eq:deltastar}, we obtain $\delta_*=1$. Since the Jordan blocks of $A_{\rm u}$ associated with the spectral boundary $\operatorname{Re}\lambda=-1$ have size $1$, by the definition of $q$ in Theorem~\ref{thm:finite-endpoint}, we obtain $q=1$. Thus \eqref{eq:q-one-endpoint} shows that the endpoint $\delta_*=1$ is achievable. In this example, a stronger conclusion holds. For $0<\delta\le1$, choose $K_\delta=(-\delta,0,0)$. Then
$$
A+BK_\delta
=
\operatorname{diag}(-\delta,-1,-2).
$$
This, along with \eqref{eq:CdeltaK}, yields
\begin{equation}\label{5.1-9-1}
C(\delta,K_\delta)
=
\sup_{t\ge0}
e^{\delta t}\norm{e^{(A+BK_\delta)t}}
=
1.
\end{equation}
Meanwhile, by \eqref{eq:CdeltaK}, we also have $C(\delta,K)\ge1$ for every feedback $K$. This, along with  \eqref{eq:Chat}
and 
\eqref{5.1-9-1}, yields
$$
\Chat(\delta)=1,
\qquad 0<\delta\le1.
$$

Thus, although $\delta_*=1$ is finite, $\Chat(\delta)$ does not blow up as
$\delta\uparrow\delta_*$; in fact, $\Chat(\delta)=1$ for all
$0<\delta\le1$.

\smallskip
\noindent\textit{(b) The order is determined by the largest Jordan-block size on the spectral boundary.
}

Let
$$
A=
\operatorname{diag}\left(
0,\,
\begin{pmatrix}
-1&0\\
1&-1
\end{pmatrix},
\,
\begin{pmatrix}
-2&0&0&0\\
1&-2&0&0\\
0&1&-2&0\\
0&0&1&-2
\end{pmatrix}
\right),
\qquad
B=e_1\in\R^7.
$$
The pair $(A,B)$ is stabilizable but not controllable, and its uncontrollable block is
$$
A_{\rm u}
=
\operatorname{diag}\left(
\begin{pmatrix}
-1&0\\
1&-1
\end{pmatrix},
\,
\begin{pmatrix}
-2&0&0&0\\
1&-2&0&0\\
0&1&-2&0\\
0&0&1&-2
\end{pmatrix}
\right).
$$
By \eqref{eq:deltastar}, we obtain $\delta_*=1$. The Jordan block associated with the spectral boundary $\operatorname{Re}\lambda=-1$ has size $2$, whereas the Jordan block of size $4$ is associated with the eigenvalue $-2$. Hence, by the definition of $q$ in Theorem~\ref{thm:finite-endpoint}, we obtain $q=2$. By Theorem~\ref{thm:finite-endpoint}, we obtain
$$
\Chat(\delta)
\asymp
(1-\delta)^{-1},
\qquad \delta\uparrow1.
$$

Thus the asymptotic order near $\delta_*$ is determined by the largest Jordan-block size associated with the spectral boundary, rather than by the largest Jordan-block size of the uncontrollable part as a whole.

\section{Appendix}
\label{sec:appendix}

\renewcommand{\thesubsection}{\Alph{subsection}}
\subsection{Proofs of the auxiliary properties in Section~\ref{sec:controllability-brunovsky}}
\label{app:properties}

\begin{proof}[Proof of Property~\ref{lem:remove-redundant-inputs}]
Since $\operatorname{Ran}B_0=\operatorname{Ran}B$, there exist
$V\in\R^{m\times r}$ and $W\in\R^{r\times m}$ such that
$B_0=BV$ and $B=B_0W$. Hence, for every
$L\in\R^{r\times n}$ and $K\in\R^{m\times n}$,
\[
    B_0L=B(VL),
    \qquad
    BK=B_0(WK).
\]
This proves \eqref{eq:closed-loop-family-B0}. Consequently,
\[
    \{A+BK:K\in\R^{m\times n}\}
    =
    \{A+B_0L:L\in\R^{r\times n}\}.
\]
Fix $\delta>0$. It follows that
\[
   \inf_{K\in\R^{m\times n}}
    \sup_{t\ge0}
    e^{\delta t}\norm{e^{(A+BK)t}}
 =
    \inf_{L\in\R^{r\times n}}
    \sup_{t\ge0}
    e^{\delta t}\norm{e^{(A+B_0L)t}}.
\]
Hence replacing $B$ by $B_0$ leaves $\Chat(\delta)$ unchanged.

Moreover, for every $j\ge0$,
\[
    \operatorname{Ran}A^jB
    =
    A^j\operatorname{Ran}B
    =
    A^j\operatorname{Ran}B_0
    =
    \operatorname{Ran}A^jB_0.
\]
Thus the spaces $\mathcal R_j$, and hence $d_j$, $\rho_j$, the
controllability indices, and $\mu$, are unchanged. Since $(A,B)$ is
controllable, $(A,B_0)$ is controllable as well.
\end{proof}

\begin{proof}[Proof of Property~\ref{prop:feedback-decay-correspondence}]
By Theorem~\ref{thm:Brunovsky-canonical-form}, the matrices $T$ and $U$ in
\eqref{eq:Br-feedback-equivalence} are invertible. Solving
\[
    \widetilde K
    =
    U^{-1}(K-K_{\rm pre})T
\]
for $K$ gives
\[
    K
    =
    K_{\rm pre}+U\widetilde K T^{-1},
\]
so $G$ is bijective and its inverse is
\eqref{eq:feedback-correspondence-inverse}. Moreover, by
\eqref{eq:Br-feedback-equivalence},
\[
\begin{aligned}
    A_{\rm Br}+B_{\rm Br}G(K)
    &=
    T^{-1}(A+BK_{\rm pre})T
    +
    T^{-1}BU\,G(K)\\
    &=
    T^{-1}(A+BK_{\rm pre})T
    +
    T^{-1}B(K-K_{\rm pre})T\\
    &=
    T^{-1}(A+BK)T,
\end{aligned}
\]
which proves \eqref{eq:closed-loop-similarity}.
\end{proof}

\begin{proof}[Proof of Property~\ref{prop:Br-reduction}]
Set
\begin{equation}\label{33-8-30}
    \kappa:=\norm{T}\norm{T^{-1}},
\end{equation}
where $T$ is as in
Property~\ref{prop:feedback-decay-correspondence}. Since $T$ is invertible,
$\kappa\ge1$. By \eqref{eq:closed-loop-similarity},
\begin{equation}\label{34-8-30}
    e^{(A_{\rm Br}+B_{\rm Br}G(K))t}
    =
    T^{-1}e^{(A+BK)t}T,
    \qquad t\ge0.
\end{equation}
Hence
\begin{equation}\label{36-8-30}
    C_{\rm Br}(\delta,G(K))
    \le
    \kappa C(\delta,K),
    \qquad \delta>0,\quad K\in\R^{r\times n},
\end{equation}
and, conversely,
\begin{equation}\label{39-8-30}
    C(\delta,K)
    \le
    \kappa C_{\rm Br}(\delta,G(K)),
    \qquad \delta>0,\quad K\in\R^{r\times n}.
\end{equation}
Taking the infimum over $K$ and using the bijectivity of $G$ from
Property~\ref{prop:feedback-decay-correspondence} gives
\[
    \frac1{\kappa}\Chat_{\rm Br}(\delta)
    \le
    \Chat(\delta)
    \le
    \kappa\Chat_{\rm Br}(\delta),
    \qquad \delta>0,
\]
which is \eqref{eq:Chat-comparison}.
\end{proof}

\subsection{Auxiliary results for the finite-endpoint analysis}
\label{app:kalman}

We give the proofs of Lemma~\ref{lem:Hu-growth} and
Proposition~\ref{prop:maximal-rate}.

\begin{proof}[Proof of Lemma~\ref{lem:Hu-growth}]
We work with the complexification of $H_{\rm u}$; the Euclidean operator
norm of a real matrix is unchanged by complexification. Choose an
invertible matrix $S$ such that $J:=S^{-1}H_{\rm u}S$ is in complex Jordan
form. By the definitions of $\delta_*$ and $q$, every eigenvalue of
$H_{\rm u}$ has nonpositive real part, and the largest Jordan block among
the eigenvalues on the imaginary axis has size $q$.

For a Jordan block $J_\lambda=\lambda I_d+N_d$,

$$
    e^{J_\lambda t}
    =
    e^{\lambda t}
    \sum_{k=0}^{d-1}\frac{t^k}{k!}N_d^k,
    \qquad t\ge0.
$$

Hence every block with $\operatorname{Re}\lambda=0$ grows at most like
$t^{q-1}$, whereas every block with $\operatorname{Re}\lambda<0$ is
bounded by an exponentially decaying polynomial. Since
$e^{H_{\rm u}t}=Se^{Jt}S^{-1}$, this gives the upper bounds in
\eqref{eq:Hu-growth} and \eqref{eq:Hu-global-growth}.

Choose a Jordan block of size $q$ associated with an eigenvalue on the
imaginary axis. For the last standard basis vector of this block, the
first component of $e^{J_\lambda t}e_q$ has modulus
$t^{q-1}/(q-1)!$. Hence
$\norm{e^{Jt}}\ge t^{q-1}/(q-1)!$. Since
$e^{Jt}=S^{-1}e^{H_{\rm u}t}S$, it follows that

$$
\norm{e^{H_{\rm u}t}}
\ge
\frac{1}{\norm{S}\norm{S^{-1}}}
\frac{t^{q-1}}{(q-1)!},
\qquad t\ge0.
$$

This gives the lower bound in \eqref{eq:Hu-growth} for all sufficiently
large $t$.

The upper bound in \eqref{eq:weighted-Hu-growth} follows from
\eqref{eq:Hu-global-growth} and the elementary estimate
$\sup_{t\ge0}e^{-\varepsilon t}t^{q-1}
\le C\varepsilon^{-(q-1)}$. If $q\ge2$, the lower bound follows from
\eqref{eq:Hu-growth} by taking $t=(q-1)/\varepsilon$ for all sufficiently
small $\varepsilon>0$. If $q=1$, the upper bound follows from
\eqref{eq:Hu-global-growth}, while the lower bound follows by taking
$t=0$. This proves \eqref{eq:weighted-Hu-growth}.
\end{proof}

\begin{proof}[Proof of Proposition~\ref{prop:maximal-rate}]
Since $(A,B)$ is stabilizable but not controllable,
Proposition~\ref{prop:uncontrollable-block}\textup{(i)} applies and shows
that $A_{\rm u}$ is Hurwitz.

We first consider the case $n_{\rm c}=0$. Then $B=0$ and $A_{\rm u}=A$,
so $C(\delta,K)=C(\delta,0)$ for every $K\in\R^{m\times n}$. If
$0<\delta<\delta_*$, the matrix $A+\delta I_n$ is Hurwitz, so
$0\in\F_\delta$. If $\delta>\delta_*$, then
$\norm{e^{At}}\ge\operatorname{spr}(e^{At})=e^{-\delta_*t}$, and hence
$C(\delta,0)=+\infty$. Finally, at $\delta=\delta_*$,
$e^{\delta_*t}e^{At}=e^{H_{\rm u}t}$, which is uniformly bounded if and
only if $q=1$ by Lemma~\ref{lem:Hu-growth}. Thus \textup{(i)--(iii)} hold
when $n_{\rm c}=0$, and \eqref{eq:maximal-rate} follows from
\textup{(i)} and \textup{(ii)}.

We may therefore assume $n_{\rm c}\ge1$. Since $(A_{\rm c},B_{\rm c})$ is
controllable, choose $K_{\rm c}$ so that
$M_{\rm c}:=A_{\rm c}+B_{\rm c}K_{\rm c}$ has spectrum in
$\operatorname{Re}z<-\delta_*-\eta$ for some $\eta>0$, and set
$K_0:=(K_{\rm c},0)Q^\top$. Then $K_0Q=(K_{\rm c},0)$, and

$$
    Q^\top(A+BK_0)Q
    =
    \begin{pmatrix}
        M_{\rm c}&A_{\rm cu}\\
        0&A_{\rm u}
    \end{pmatrix}.
$$

If $0<\delta<\delta_*$, both diagonal blocks of
$Q^\top(A+BK_0)Q+\delta I_n$ are Hurwitz. Hence the whole block upper
triangular matrix is Hurwitz, and therefore $K_0\in\F_\delta$. This proves
\textup{(i)}.

For $\delta>\delta_*$, \eqref{eq:block-lower-bound} and
$\norm{e^{A_{\rm u}t}}\ge
\operatorname{spr}(e^{A_{\rm u}t})=e^{-\delta_*t}$ give
$\sup_{t\ge0}e^{\delta t}\norm{e^{(A+BK)t}}=+\infty$ for every feedback
$K$. Thus $\F_\delta=\varnothing$, proving \textup{(ii)}.

At $\delta=\delta_*$, since $(A,B)$ is stabilizable but not controllable,
Lemma~\ref{lem:Hu-growth} applies. Moreover,
$e^{\delta_*t}e^{A_{\rm u}t}=e^{H_{\rm u}t}$, and by that lemma this
semigroup is uniformly bounded if and only
if $q=1$. Hence \eqref{eq:block-lower-bound} shows that $q=1$ is necessary
for $\F_{\delta_*}\neq\varnothing$.

Assume $q=1$ and set
$G_{\rm c}:=M_{\rm c}+\delta_*I_{n_{\rm c}}$. Then
$\norm{e^{G_{\rm c}t}}\le M_{\rm c}^{(0)}e^{-\eta t}$ for some
$M_{\rm c}^{(0)}>0$, while Lemma~\ref{lem:Hu-growth}, in particular
\eqref{eq:Hu-global-growth}, gives
$\norm{e^{H_{\rm u}t}}\le M_{\rm u}^{(0)}$ for some
$M_{\rm u}^{(0)}>0$. The upper-right block of the shifted semigroup is

$$
    \int_0^t e^{G_{\rm c}(t-s)}A_{\rm cu}e^{H_{\rm u}s}\,ds,
    \qquad t\ge0,
$$

whose norm is uniformly bounded for $t\ge0$. Therefore
$e^{\delta_*t}e^{(A+BK_0)t}$ is uniformly bounded for $t\ge0$, so
$K_0\in\F_{\delta_*}$. This proves \textup{(iii)}. Finally,
\eqref{eq:maximal-rate} follows immediately from \textup{(i)} and
\textup{(ii)}.
\end{proof}


\begin{thebibliography}{99}

\bibitem{ApkarianNoll2025}
P.~Apkarian and D.~Noll,
\newblock Minimizing transients via the Kreiss system norm,
\newblock \emph{Journal of the Franklin Institute}, 362 (2025), 108208.

\bibitem{Brunovsky1970}
P.~Brunovsk\'y,
\newblock A classification of linear controllable systems,
\newblock \emph{Kybernetika}, 6 (1970), 173--188.

\bibitem{Hautus1970}
M.~L.~J.~Hautus,
\newblock Stabilization, controllability and observability of linear
autonomous systems,
\newblock \emph{Indagationes Mathematicae (Proceedings)}, 73 (1970),
448--455.

\bibitem{HPW2002}
D.~Hinrichsen, E.~Plischke, and F.~Wirth,
\newblock State feedback stabilization with guaranteed transient bounds,
\newblock in \emph{Proceedings of the Symposium on Mathematical Theory of
Networks and Systems (MTNS 2002)}, South Bend, Indiana, 2002, paper 2132.

\bibitem{Kalman1963}
R.~E.~Kalman,
\newblock Mathematical description of linear dynamical systems,
\newblock \emph{Journal of the Society for Industrial and Applied Mathematics,
Series A: Control}, 1 (1963), 152--192.

\bibitem{LiZhouCao2026}
X.~Li, B.~Zhou, and J.~Cao,
\newblock On bounds of the solution to the parametric Lyapunov equations with applications,
\newblock \emph{Applied Mathematics and Computation}, 531 (2026), 130172.

\bibitem{PW2004}
E.~Plischke and F.~Wirth,
\newblock Stabilization of linear systems with prescribed transient bounds,
\newblock in \emph{Proceedings of the 16th International Symposium on
Mathematical Theory of Networks and Systems (MTNS 2004)}, Leuven, 2004.

\bibitem{Seidman1988}
T.~I.~Seidman,
\newblock How violent are fast controls?,
\newblock \emph{Mathematics of Control, Signals, and Systems}, 1 (1988),
89--95.

\bibitem{SeidmanYong1996}
T.~I.~Seidman and J.~Yong,
\newblock How violent are fast controls?, II,
\newblock \emph{Mathematics of Control, Signals, and Systems}, 9 (1996),
327--340.

\bibitem{SWXY2006}
Y.~G.~Sun, L.~Wang, G.~Xie, and M.~Yu,
\newblock Improved overshoot estimation in pole placements and its
application in observer-based stabilization for switched systems,
\newblock \emph{IEEE Transactions on Automatic Control}, 51 (2006),
1962--1966.

\bibitem{WhidborneMcKernan2007}
J.~F.~Whidborne and J.~McKernan,
\newblock On the minimization of maximum transient energy growth,
\newblock \emph{IEEE Transactions on Automatic Control}, 52 (2007),
1762--1767.

\bibitem{WonhamMorse1972}
W.~M.~Wonham and A.~S.~Morse,
\newblock Feedback invariants of linear multivariable systems,
\newblock \emph{Automatica}, 8 (1972), 93--100.

\end{thebibliography}
\end{document}